\documentclass[11pt]{amsart}
\usepackage{amsmath, amsrefs, amsthm, amssymb}
\newcommand\R{{\mathbb{R}}}

\newcommand\N{{\mathbb{N}}}
\newcommand\C{{\mathbb{C}}}

\newtheorem{theorem}{Theorem}[section]
\newtheorem{proposition}[theorem]{Proposition}

\newtheorem{corollary}[theorem]{Corollary}

\theoremstyle{definition}
\newtheorem{definition}[theorem]{Definition}

\theoremstyle{remark}
\newtheorem{remark}[theorem]{Remark}

\numberwithin{equation}{section}

\def\R{\mathbb{R}}
\def\C{\mathbb{C}}

\begin{document}

\title[Breakdown for Camassa--Holm]
{Breakdown for the Camassa--Holm Equation Using\\
Decay Criteria and  Persistence in Weighted Spaces}

\author{Lorenzo Brandolese}

\address{L. Brandolese: Universit\'e de Lyon~; Universit\'e Lyon 1~;
CNRS UMR 5208 Institut Camille Jordan,
43 bd. du 11 novembre,
Villeurbanne Cedex F-69622, France.}
\email{brandolese{@}math.univ-lyon1.fr}
\urladdr{http://math.univ-lyon1.fr/$\sim$brandolese}


\date{October 17, 2011}


\begin{abstract}
We exhibit a sufficient condition in terms of decay at infinity of the initial data
for the finite time blowup of strong solutions to the Camassa--Holm equation:
a wave breaking will occur as soon as the initial data decay faster at infinity than the solitons.
In the case of data decaying slower than solitons
we provide persistence results for the solution
in weighted $L^p$-spaces, for a large class of moderate weights.
Explicit asymptotic profiles illustrate the optimality of these results.
 \end{abstract}

\maketitle

\begin{center}
The original publication is published by \emph
{International Mathematics Research Notices} (Oxford University Press).
Doi:10.1093/imrn/rnr218
\end{center}


\bigskip
\section{Introduction}

In this paper we consider the Camassa--Holm equation on~$\R$, 
\begin{equation}
\label{CH}
 \partial_t u+u\partial_x u=
 P(D)\Bigl(u^2+\textstyle\frac{1}{2}(\partial_x u)^2\Bigr),
\end{equation}
where
\begin{equation}
 P(D)=-\partial_x(1-\partial_x^2)^{-1}
\end{equation}
and $t,x\in\R.$

For smooth solutions, equation~\eqref{CH} can be also rewritten in the more usual form
\begin{equation}
\label{CH1}
\partial_t u-\partial_t\partial^2_x u +3u\partial_x u-2\partial_x u \partial^2_x u -u\partial_x^3 u=0.
\end{equation}
The Camassa--Holm equation arises approximating the Hamiltonian for the Euler's equation
in the shallow water regime. It is now a popular model for the propagation of unidirectional water waves
over a flat bed. Its hydrodynamical relevance has been pointed out in~\cite{CamH93}, \cite{AConL09}.
In this setting, $u(x,t)$ represents the horizontal velocity of the fluid motion at a certain depth.
Interesting mathematical 
properties of such equation include its bi-Hamiltonian structure, the existence of infinitely many
conserved integral quantities (see~\cite{FokF81}), and its geometric interpretation
in terms of geodesic flows on the diffeomorphism group (see~\cite{ACon00}, \cite{Kou99}).

Another important feature of the Camassa--Holm equation is the existence of traveling 
solitary waves, interacting like solitons, also called ``peakons" due to their shape (see \cite{CamH93}):
\begin{equation*}
u_c(x,t)=c\,e^{-|x-ct|}.
\end{equation*}
The peakons replicate a feature that is characteristic for the waves of great height, waves of largest amplitude
that are exact solutions of the governing equations for irrotational water waves (see \cite{AConE07}, \cite{Tol96}).
These solutions are orbitally stable: their shape is stable under small perturbations and therefore these waves 
are physically recognizable, see \cite{AConS00}.

Let us recall that the Cauchy problem associated with equation~\eqref{CH}
is locally well-posed  (in the sense of Hadamard) in $H^s(\R)$ for $s>3/2$.
See, for example, \cite{Dan01}, \cite{Dan03} (see also \cite{DeLKT11} for the periodic case), where
the well-posedness is also established in slightly rougher spaces of Besov type.

There exist strong solutions to the Camassa--Holm equation that exist globally
and strong solutions that blow up in finite time.
Several criteria can be given in order to establish whether or not an initially smooth solution develops
a singularity after some  time.
For example, A.~Constantin and J.~Escher showed that the solution must blow up when the initial datum
is odd with negative derivative at the origin, or when $u_0$ has in some point a slope that is less then
$-\frac{1}{\sqrt2}\|u_0\|_{H^1}$. 
The same authors showed that the solution is global in time provided that
the potential $u_0-\partial^2_xu_0$ associated with the initial value~$u_0$ is a bounded measure with constant sign.
These and many other important results have been surveyed by L.~Molinet in~\cite{Mol04}. 
See also~\cite{ACon00}, \cite{HMPZ07}, \cite{Yin04}.
The criterion on the potential  was later improved by McKean~\cite{McKean04},
who actually gave a necessary and sufficient condition for the global existence of the solution~$u$, in terms
of conditions on the sign of $u_0-\partial^2_xu_0$.

The blowup mechanism of strong solutions is quite well understood:
such solutions remain uniformly bounded up to the blowup time and 
a blowup occurs at a time $T^*<\infty$ 
if and only if $\lim_{t\uparrow T^*}\{ \inf_{x\in\R}\partial_xu (t,x)\}=-\infty$.
This phenomenon is also referred as wave breaking.

For $s<3/2$, the well-posedness of equation~\eqref{CH} in~$H^s$ is not known
(only counterexamples to the uniform well-posedness are available, see
\cite{HimM01}).
However, existence and uniqueness results of weak solutions
have been recently established also in spaces with low regularity. See, for example,
\cite{BreC07}, \cite{BreC07bis}, \cite{XinZ00}, \cite{XinZ02}
for the construction of weak solutions arising from~$H^1$-data.
Such studies are relevant also in connection with  the continuation problem of solutions after the breakdown.

\medskip
\paragraph{\bf Main contributions.\,}

One goal of this paper is to study the persistence and non-persistence
properties of solutions $u\in C([0,T],H^s)$ in  weighted spaces.
By the Sobolev embedding theorem, such solutions satisfy
$\sup_{t\in[0,T]}  \Bigl( \| u(t)\|_p+\|\partial_xu(t)\|_p \Bigr) <\infty$ for all $2\le p\le\infty$.
We will determine a wide class of weight functions~$\phi$ such that
 \[
 \sup_{t\in[0,T]}  \Bigl( \| u(t)\phi\|_p+\|\partial_xu(t)\phi\|_p \Bigr) <\infty.
\]
The analysis of the solutions in weighted spaces is useful to obtain information
on their spatial asymptotic behavior. 
A second motivation of such analysis is provided by the paper of Himonas, Misio\l ek, Ponce and Zhou \cite{HMPZ07},
and by the earlier work of  Escauriaza, Kenig, Ponce and Vega, 
\cite{EKPV1}, \cite{EKPV2}
on the unique continuation of solutions to nonlinear dispersive equations (the Schr\"odinger and
the generalized $k$-generalized KdV equations):
indeed, studying the spatial properties of the solution~$u$ at two different times provides simple
criteria for $u$ to be identically zero. 

In the case of the Camassa--Holm equations,
only exponential weights $\phi(x)=e^{a|x|}$ with $0<a<1$ (see~\cite{HMPZ07}),
or power weights $\phi(x)=(1+|x|)^{c}$ with $c>0$ (see~\cite{NiZ11})
have been considered so far in the literature.
Moreover, such analysis was performed only in a weighted-$L^\infty$ setting:
our class of weights allows us to give a unified approach
and to cover also the case $2\le p<\infty$.

The novel feature of our approach will be the systematic use of {\it moderate weight functions\/}.
Moderate weights are a commonly used tool in time-frequency analysis,
see, for example,~\cite{AldG01}, \cite{Gro07},
but they have received little attention in connection
with PDEs, despite their rich harmonic analysis properties.
As we shall see later on, in the case of the Camassa--Holm equations they lead to
optimal results  (see Theorem~\ref{theo2}, Theorem~\ref{cor1} and its corollary).
Restricting to moderate weights is natural. Indeed, the Camassa--Holm equation is 
invariant under spatial translations, and the weighted spaces 
$L^p_\phi(\R):=L^p(\R,\phi(x)^p\,dx)$ are shift invariant if and only if $\phi$ is a moderate weight.

\medskip
The second purpose of this paper is to establish the following blowup criterion:
if the initial datum $u_0\not\equiv0$ is such that 
 \begin{equation}
  \liminf_{|x|\to\infty} \,e^{|x|}\bigl( |u_0(x)| + |\partial_xu_0(x)|\bigr)=0,
 \end{equation}
then a wave breaking must occur at some time $T^*<\infty$.
In particular, a breakdown occurs when the initial datum satisfies
\[ \Bigl \|e^{|x|} \bigl(|u_0|+|\partial_x u_0|\bigr)\, \Bigr\|_{L^p}<\infty \]
for some $1\le p<\infty$.

 The exponential decay assumption for $u_0$ and its derivative looks optimal, as the peakons
 $u_c(x,t)=ce^{-|x-ct|}$ are well-known global solutions to~\eqref{CH}.
 This second result improves the conclusion  of~\cite{HMPZ07}, where
 the authors showed that compactly supported data lead to a breakdown.


\bigskip

\section{Analysis of the Camassa--Holm equation in weighted spaces}

Let us first recall some standard definitions. 
In general a weight function is simply a non-negative function.
A weight function~$v\colon\R^n\to\R$ is {\it sub-multiplicative\/} if
\[
 v(x+y)\le v(x)v(y),\qquad\forall x,y\in\R^n.
\]

Given a sub-multiplicative function~$v$, by definition 
a positive function $\phi$ is {\it $v$-moderate\/} if and only if
\begin{equation}
 \label{conm}
 \exists\,C_0>0\colon\;
 \phi(x+y)\le C_0\,v(x)\phi(y),\qquad\forall x,y\in\R^n.
\end{equation}
If $\phi$ is $v$-moderate for some sub-multiplicative function~$v$ then we say that~$\phi$
is {\it moderate\/}.
This is the usual terminology in time-frequency analysis papers, \cite{AldG01}, \cite{Gro07}.
Let us recall the most standard examples of such weights. Let
\begin{equation}
 \label{exaw}
 \phi(x)=\phi_{a,b,c,d}(x)=e^{a|x|^b}(1+|x|)^{c}\log(e+|x|)^d.
\end{equation}

We have (see~\cite{Fei79}, \cite{Gro07}):
\begin{itemize}
\item
For $a,c,d\ge0$ and $0\le b\le 1$ such weight is sub-multiplicative.
\item
If $a,c,d\in\R$ and $0\le b\le1$ then $\phi$ is moderate.
More precisely, $\phi_{a,b,c,d}$ is $\phi_{\alpha,\beta,\gamma,\delta}$-moderate
for $|a|\le \alpha$, $b\le\beta$, $|c|\le\gamma$ and $|d|\le\delta$.

\end{itemize}

We will now specify the class of admissible weight functions.

\begin{definition}
\label{defaw}
An {\it admissible weight function} for the Camassa--Holm equation
is a locally absolutely continuous function $\phi\colon\R\to\R$ such that for some $A>0$ and
a.e.~$x\in\R$,  $|\phi'(x)|\le A|\phi(x)|$, and that is $v$-moderate,
for some sub-multiplicative weight function~$v$ satisfying $\inf_{\R}v>0$ and
\begin{equation}
 \label{vadm}
 \int_\R \frac{v(x)}{e^{|x|}} \,dx<\infty.
\end{equation}
\end{definition}

\medskip
We recall  that a locally absolutely continuous function
is a.e. differentiable in~$\R$. Moreover, its a.e. derivative belongs to $L^1_{\rm loc}$
and agrees with its distributional derivative.


We can now state our main result on admissible weights.

\begin{theorem}
 \label{theo2}
 Let $T>0$, $s>3/2$ and $2\le p\le\infty$.
 Let also $u\in C([0,T],H^s(\R))$ be a strong solutions
 of the Cauchy problem for  equation~\eqref{CH}, such that
 $u|_{t=0}=u_0$ satisfies
  \[
  u_0\,\phi\in L^p(\R)\qquad\hbox{and}\qquad
  (\partial_xu_0)\phi\in L^p(\R),
 \]
where $\phi$ is an admissible weight function for the Camassa--Holm equation.
Then, for all $t\in[0,T]$ we have the estimate,
\[
  \| u(t)\phi\|_p+\|\partial_xu(t)\phi\|_p
  \le \Bigl(\| u_0\phi\|_p+\|\partial_xu_0\phi\|_p\Bigr)e^{CMt}, 
 \]
for some constant $C>0$ depending only on $v$, $\phi$ 
(through the constants $A$, $C_0$,  $\inf_\R v$ and $\int \frac{v(x)}{e^{|x|}}\,dx$), and
\[
M\equiv\,\sup_{t\in[0,T]}  \Bigl(\|u(t)\|_\infty+\|\partial_x u(t)\|_\infty\Bigr)<\infty.
\]
\end{theorem}

\medskip

\begin{remark}
\label{rem1}
The basic example of application of Theorem~\ref{theo2} is obtained
by choosing the standard weights $\phi=\phi_{a,b,c,d}$ as in~\eqref{exaw}
with the following conditions
\[
a\ge 0,\qquad c,d\in\R, \qquad 0\le b\le 1,\qquad ab<1.  
\]
(for $a<0$, one has $\phi(x)\to0$ as $|x|\to\infty$: the conclusion of the theorem remains 
true but it is not interesting in this case).
The restriction~$ab<1$ guarantees the validity of condition~\eqref{vadm}
for a multiplicative function $v(x)\ge1$.

The limit case $a=b=1$ is not covered by Theorem~\ref{theo2}.
The result holds true, however, for the weight $\phi=\phi_{1,1,c,d}$  with $c<0$, $d\in\R$
and $\frac{1}{|c|}<p\le\infty$, or more generally when 
$(1+|\!\cdot\!|)^c\log(e+|\!\cdot\!|)^d\in L^p(\R)$.
See Theorem~\ref{cor1} below, which covers the case of such fast growing weights.
\end{remark}

\begin{remark}
Let us consider a few particular cases:
\begin{enumerate}
 \item Take $\phi=\phi_{0,0,c,0}$ with $c>0$, and choose $p=\infty$.
 In this case the Theorem~\ref{theo2} states that
 the condition
 \[
  |u_0(x)|+|\partial_xu_0(x)|\le C(1+|x|)^{-c}
 \]
implies the uniform algebraic decay in $[0,T]$:
\[
  |u(x,t)|+|\partial_xu(x,t)|\le C'(1+|x|)^{-c}.
 \]
Thus, we recover the main result of~\cite{NiZ11}.

\item
Choose $\phi=\phi_{a,1,0,0}(x)$ if $x\ge0$ and $\phi(x)=1$ if $x\le0$ 
with $0\le a<1$. Such weight clearly satisfies the admissibility conditions of Definition~\ref{defaw}.
Applying Theorem~\ref{theo2} with $p=\infty$ we conclude that  the
pointwise decay $O(e^{-ax})$ as $x\to+\infty$ 
is conserved during the evolution.
Similarly, we have persistence of the decay $O(e^{-ax})$ as $x\to-\infty$.
Hence, our Theorem~\ref{theo2} encompasses also Theorem~{1.2} of~\cite{HMPZ07}. 
\end{enumerate}
\end{remark}

\bigskip
\section{Elementary properties of sub-multiplicative and moderate weights}

We collect in this section some basic facts on moderate weights.
We implicitly assume the continuity of the weights $v$ and $\phi$.

The weighted space $L^p_\phi(\R)=L^p(\R,\phi(x)^p\,dx)$
is translation invariant if and only if $\phi$ is a moderate weight.
Indeed, if $\phi$ is $v$-moderate one clearly has 
$\|f(\cdot-y)\|_{L^p_\phi}\le C_0v(y)\|f\|_{L^p_\phi}$.
Conversely, if $1\le p<\infty$ and if $L^p_\phi(\R)$ is translation invariant, then
one easily checks that
\[
 v(x):=\sup_{\|f\|_{L^p_\phi}\le1}\|f(\cdot-x)\|_{L^p_\phi}
\]
is sub-multiplicative and that $\phi$ is $v$-moderate.
See also~\cite{AldG01} for more details.

The interest of imposing the sub-multiplicativity condition on a weight function
is also made clear by the following proposition:

\begin{proposition}
\label{wyoung}
 Let $v\colon\R^n\to\R^+$ and $C_0>0$.
 Then following conditions are equivalent:
 \begin{enumerate}
  \item 
  $ \forall\,x,y\colon\;v(x+y)\le  C_0\,v(x)v(y)$.
 \item
 For all $1\le p,q,r\le\infty$ and for 
 any measurable functions $f_1,f_2\colon\R^n\to\C$ the weighted Young inequalities hold:
 \[
  \|(f_1*f_2)v\|_r\le C_0\|f_1v\|_p\|f_2v\|_q,
  \qquad 1+\frac1r=\frac1p+\frac1q.
 \]
 \end{enumerate}
\end{proposition}

\begin{proof}
 To see that the first claim implies the second one, one writes $v(x)\le C_0v(x-y)v(y)$ and then applies
 the classical Young inequality to $|f_1|\phi$ and $|f_2|\phi$.
 
 Conversely, take $p=1$ and $q,r=\infty$.
 We can assume that $v$ does not vanish.
 Let $R(x,y)=\frac{v(x)}{v(x-y)v(y)}$.
 Then, for all $h\in L^1(\R^n)$ we have
 \begin{equation*}
  \begin{split}
\biggl|\int h(y)R(x,y)\,dy\biggr| &=\biggl|\int h(x-y)R(x,y)\,dy\biggr|\le C_0\|h\|_1 ,  
  \end{split}
 \end{equation*}
where in the last inequality we applied the weighted Young estimate with $p=1$, $q,r=\infty$,
$f_1=\frac{h}{v}$ and $f_2=\frac{1}{v}$.
By duality this implies $R(x,y)\le C_0$, that is, $v$ satisfies the sub-multiplicativity inequality~(1).
\end{proof}

The moderateness of a weight function is the good condition for
weighed Young inequalities with two different weights.

\begin{proposition}
 \label{wyoun2}
Let $1\le p\le\infty$ and $v$ be a sub-multiplicative weight on~$\R^n$.
The two following two conditions are equivalent:
\begin{enumerate}
 \item $\phi$ is a $v$-moderate weight function (with constant $C_0$).
 \item For all measurable functions $f_1$ and $f_2$ the weighted
Young estimate holds
\begin{equation}
 \label{y1p}
 \|(f_1*f_2)\phi\|_p\le C_0\|f_1v\|_1 \,\|f_2\phi\|_p.
\end{equation}
\end{enumerate}
\end{proposition}

\begin{proof}
 The proof is similar to the previous one and can be found in~\cite{Fei79}, but we give it
 for the reader's convenience.
 It is obvious that the first claim implies the second one.
 Conversely, for all $h_1=f_1v\in L^1(\R)$ and $h_2=f_2\phi\in L^p(\R)$ we have, by assumption,
 \[
 \int\biggl| \int h_1(x-y)h_2(y)R(x,y)\,dy\biggr|^p\,dx\le C_0^p\|h_1\|_1^p\|h_2\|_p^p,
 \]
 with $R(x,y)=\frac{\phi(x)}{v(x-y)\phi(y)}$.
Now choose $h_1=h_{1,n}=n{\bf 1}_{[x_0-\frac{1}{2n},x_0+\frac{1}{2n}]}$, for some $x_0\in\R$ and $n=1,2,\ldots$,
next  choose $h_2$ continuous and compactly supported.
The integral on the left hand side equals
\[
\int n^p \biggl| \int_{x-x_0-\frac{1}{2n}}^{x-x_0+\frac{1}{2n}} h_2(y)R(x,y)\,dy\biggr|^p\,dx.
\]
By the continuity of $y\mapsto h_2(y)R(x,y)$ and the 
dominated convergence theorem,  letting $n\to\infty$ we get
\[
\int |h_2(x-x_0)|^p R(x,x-x_0)^p\,dx\le C_0^p\|h_2\|_p^p,
\]
for all $h_2\in C_c(\R)$.
We deduce by a duality argument that  $R(x,x-x_0)\le C_0$,
for all $x_0\in\R$, that is, $\phi$ is $v$-moderate with constant $C_0$. 
\end{proof}

\medskip

We finish this section by listing further elementary properties
of sub-multiplicative and moderate  weights.
Even though such properties will not be needed in the sequel,
they  shed some light on Definition~\ref{defaw}.
We assume as usual the continuity of $v$ and $\phi$.

\begin{enumerate}
 \item
 If $v\not\equiv0$ is an {\it even\/} sub-multiplicative weight function, 
 then $\inf_R v\ge1$.
\item
Every nontrivial sub-multiplicative or moderate weight grows and decays not faster than
exponentially:
there exists $a\ge0$ such that
\[
 e^{-a}e^{-a|x|}\le \phi(x)\le e^{a}e^{a|x|}.
\]
\item
\label{ederiv}
Let $\phi$ be a locally absolutely continuous
$v$-moderate weight such that $C_0v(0)=1$ (where $C_0$ is the constant
in ~\eqref{conm}).
If $v$ has both left and right derivatives at the origin, then 
for a.e. $y\in\R$, 
\[
 |\phi'(y)|\le A\,\phi(y). 
\]
with $A=C_0\max\{|v'(0-)|,|v'(0+)|\}$.
\end{enumerate}

In fact, if $v'(0+)$ and $\phi'(y+)$ exist, then
$
 \phi'(y+)\le C_0v'(0+)\phi(y)
$,
and if $v'(0-)$ and $\phi'(y-)$ exist, then
$
 \phi'(y-)\ge C_0v'(0-)\phi(y)
$.

We leave  to the reader the simple verification of the first two properties.

\bigskip
\section{Proof of Theorem~\ref{theo2}}

\begin{proof}[Proof of Theorem~\ref{theo2}]
We denote $F(u)=u^2+\frac12 (\partial_x u)^2$.
We also introduce the kernel
\begin{equation*}
  G(x)=\textstyle\frac{1}{2}e^{-|x|}.
\end{equation*}
Then the Camassa--Holm equation~\eqref{CH} can be rewritten as
\begin{equation}
\label{CH2}
 \partial_t u+u\partial_x u
 +\partial_xG*F(u)=0,
\end{equation}

Notice that from the assumption $u\in C([0,T],H^{s})$, $s>3/2$, we have
\[
M\equiv \sup_{t\in[0,T]} \Bigl( \|u(t)\|_\infty+\|\partial_x u(t)\|_\infty\Bigr)<\infty.
\]

For any  $N\in \N\backslash\{0\}$, let us consider the $N$-truncations    
\[
 f(x)=f_N(x)=\min\{\phi(x),N\}.
\]
Observe that $f\colon\R\to\R$ is a locally absolutely continuous function, such that
\[ 
 \|f\|_\infty\le N, \qquad |f'(x)|\le A|f(x)| \quad\hbox{a.e.}
 \]
In addition, if $C_1=\max\{C_0,\alpha^{-1}\}$, where $\alpha=\inf_{x\in\R} v(x)>0$, then
\begin{equation*}
 f(x+y)\le C_1\,v(x)f(y),\qquad\forall x,y\in\R.
\end{equation*}
Indeed, let us introduce the set $U_N=\{x\colon \phi(x)\le N\}$:
if $y\in U_N$, then
$f(x+y)\le \phi(x+y)\le C_0v(x)f(y)$;
if $y\not\in U_N$, then $f(x+y)\le N=f(y)\le \alpha^{-1}v(x)f(y)$.

The constant~$C_1$ being independent on~$N$, this shows that the $N$-truncations
of a $v$-moderate weight are uniformly $v$-moderate with respect to~$N$.

\medskip
We start considering the case $2\le p<\infty$. 
Multiplying the equation~\eqref{CH2} by $f$, next by $|uf|^{p-2}(uf)$
we get, after integration,
\begin{equation}
 \label{din1}
 \frac{1}{p}\, \frac{d}{dt}\Bigl( \|uf\|_{p}^{p}  \Bigr)
 +\int |uf|^{p}(\partial_x u)\,dx
 +\int |uf|^{p-2}(uf)(f\partial_xG*F(u))\,dx=0.
\end{equation}
The two above integrals are clearly finite since $f\in L^\infty(\R)$ and $u(\cdot,t)\in H^s$, $s>3/2$,
hence $\partial_xG*F(u)\in L^1\cap L^\infty$.
Estimating the absolute value of the first integral by $M\|uf\|_{p}^{p}$ and of the second one
by $\|uf\|_{p}^{p-1}\|f(\partial_x G*F(u))\|_{p}$,
we get
\begin{equation}
\label{1pe}
\begin{split}
\frac{d}{dt}\|uf\|_{p} 
 &\le M\|uf\|_{p}+\|f(\partial_xG*F(u))\|_{p}\\
 &\le M\|uf\|_{p}+C_1\|(\partial_xG) v\|_1\|F(u)f\|_{p}\\
 &\le M\|uf\|_{p}+C_2\|F(u)f\|_{p}\\
 &\le M\,\|uf\|_{p}+MC_2\|(\partial_x u)f\|_{p}\,.
\end{split}
\end{equation}
In the second inequality we applied Proposition~\ref{wyoun2},
and in the third we used the pointwise bound
$|\partial_xG(x)|\le \frac{1}{2}e^{-|x|}$ and we applied 
condition~\eqref{vadm}. Here, $C_2$ depends only on~$v$ and~$\phi$.

We now look for a similar inequality for $(\partial_x u)f$.
Differentiating equation~\eqref{CH2} with respect to~$x$, next multiplying by~$f$
we obtain
\begin{equation}
\label{eqmu}
\partial_t[(\partial_x u)f]+uf\partial_x^2u+[(\partial_x u)f](\partial_x u)+f(\partial_x^2G)*F(u)=0.
\end{equation}

Before multiplying this equation by~$|(\partial_x u)f|^{p-2}(\partial_x u)f$ and integrating, let us 
study the second term. Observe that
\begin{equation*}
 \begin{split}
  &\int uf(\partial_x^2 u)\,|(\partial_x u)f|^{p-2}(\partial_x u)f\,dx \\
  &\qquad=\int u |(\partial_x u)f|^{p-2}(\partial_x u)f  
    \bigl[\partial_x\bigl((\partial_x u)f\bigr)-(\partial_xu)(\partial_xf)\bigr]\,dx\\
  &\qquad=\int u\,\partial_x\biggl(\frac{|(\partial_xu)f|^{p}}{p}\biggr)
    -\int u|(\partial_xu)f\bigr|^{p-2}(\partial_xu)f(\partial_x u)(\partial_x f)\,dx.
 \end{split}
\end{equation*}
But $|\partial_x f(x)|\le Af(x)$ for a.e.~$x$, then it follows that
\begin{equation*}
 \begin{split}
 &\biggl| \int uf(\partial_x^2 u)\,|(\partial_x u)f|^{p-2}(\partial_x u)f\,dx \biggr| \\
 &\qquad 
 \le p^{-1} \|\partial_xu\|_\infty\|(\partial_xu)f\|_{p}^{p}
  +A\|u\|_\infty \|(\partial_x u)f\|_{p}^{p}\\
  &\qquad \le M(1+A)\|(\partial_x u)f\|_{p}^{p} .
  \end{split}
\end{equation*}

Now multiplying~\eqref{eqmu} by $|(\partial_x u)f|^{p-2}(\partial_xu)f$,
integrating and using the last estimate, we get (arguing as before)
\begin{equation*}
 \label{pee}
 \begin{split}
 \frac{d}{dt}\|(\partial_xu)f\|_{p}
  &\le M(2+A)\,\|(\partial_xu)f\|_{p}
 +\|f(\partial_x^2G)*F(u)\|_{p}.\\
\end{split}
\end{equation*}
The two functions $(\partial_xu)f$ and $\partial_xu$  belong to $L^2\cap L^\infty$.
Moreover, the kernel $G$ satisfies $\partial_x^2G=G-\delta$, hence we have
$f*(\partial^2_xG*F(u))\in L^1\cap L^\infty$.

In addition, the identity $\partial_x^2G=G-\delta$ and Proposition~\ref{wyoun2}
with condition~\eqref{vadm} imply that
\begin{equation*}
\begin{split}
 \|f(\partial_x^2G*F(u))\|_{p} &\le C_3\|F(u)f\|_{p}\\
 &\le C_3M\bigl(\|uf\|_{p}+\|(\partial_x u)f\|_{p}\bigr).
\end{split}
\end{equation*}
Hence,
\begin{equation}
 \label{2pe}
\frac{d}{dt}\|(\partial_xu)f\|_{p} 
\le C_4M \bigl(\|uf\|_{p}+\|(\partial_x u)f\|_{p}\bigr).
\end{equation}

Now, summing the inequalities~\eqref{1pe}-\eqref{2pe} and then integrating yields,
for all $t\in[0,T]$,
\begin{equation*}
 \begin{split}
 \|u(t)f\|_{p}+\|(\partial_xu)(t)f\|_{p}
  &\le \Bigl( \|u_0f\|_{p}+\|(\partial_xu_0)f\|_{p}\Bigr) \exp\bigl(CMt\bigr).
 \end{split}
\end{equation*}
Here $C$, $C_3$ and $C_4$ are positive constants depending only on~$\phi$ and $v$.
But for a.e.~$x\in\R$, $f(x)=f_N(x)\uparrow\phi(x)$ as $N\to\infty$.
As by our assumptions $u_0\phi\in L^p(\R)$ and $(\partial_xu_0)\phi\in L^p(\R)$,  we arrive at
\begin{equation}
\label{fine}
\begin{split}
 \|u(t)\phi\|_{p}+\|(\partial_xu)(t)\phi\|_{p}
  &\le \Bigl( \|u_0\phi\|_{p}+\|(\partial_xu_0)\phi\|_{p}\Bigr) \exp\bigl(CMt\bigr).
 \end{split}
\end{equation}



\medskip
It remains to treat the case $p=\infty$.
We have $u_0,\partial_x u_0\in L^2\cap L^\infty$ and $f=f_N\in L^\infty$.
Hence, for all $2\le q<\infty$ we have as before
\begin{equation*}
 \begin{split}
 \|u(t)f\|_{q}+\|(\partial_xu)(t)f\|_{q}
  &\le \Bigl( \|u_0f\|_{q}+\|(\partial_xu_0)f\|_{q}\Bigr) \exp\bigl(CMt\bigr).\\
 \end{split}
\end{equation*}
But the last factor in the right-hand side is independent on~$q$. 
Then letting  $q\to\infty$ and using the well-known fact that the $L^\infty$-norm
is the limit (possibly $=+\infty$) of $L^q$ norms as $q\to\infty$, implies that
\begin{equation*}
 \begin{split}
 \|u(t)f\|_{\infty}+\|(\partial_xu)(t)f\|_{\infty}
  &\le \Bigl( \|u_0f\|_{\infty}+\|(\partial_xu_0)f\|_{\infty}\Bigr) \exp\bigl(CMt\bigr).\\
 \end{split}
\end{equation*}
The last factor in the right-hand side is independent on~$N$. 
Now taking $N\to\infty$ implies
that estimate~\eqref{fine} remains valid for $p=\infty$.

\end{proof}


\section{Fast growing weights and exact asymptotic profiles}

As we observed in Remark~\ref{rem1}, Theorem~\ref{theo2}
does not cover some limit cases of fast growing weights.
The purpose of this section is to establish a variant of this theorem 
that can be applied to some $v$-moderate weights~$\phi$ for which
condition~\eqref{vadm} \emph{does not hold}.
We recall that condition~\eqref{vadm} reads
\[
 \int \frac{v(x)}{e^{|x|}}\,dx<\infty.
\]

Let $2\le p\le\infty$.
Instead of assuming~\eqref{vadm}, we will now put the weaker condition
\begin{equation}
\label{vadmp}
ve^{-|\cdot|} \in L^p(\R).
\end{equation}
It easily checked that for any continuous sub-multiplicative weight function~$v$ we have
\[
ve^{-|\cdot|}\in L^1(\R)\;\Longrightarrow\; ve^{-|\cdot|}\in L^p(\R)\quad \forall\,1\le p\le \infty,
\]
so that condition~\eqref{vadmp} is indeed weaker than condition~\eqref{vadm}.
Indeed, let $h=e^{-|\cdot|}v\in L^1(\R)$. The open set
$\{x\colon h(x)>1\}$ is the (possibly empty) disjoint union of open bounded intervals $I_n=(a_n,b_n)$ and $h(a_n)=h(b_n)=1$.
For sufficiently large~$n$, $I_n$ in included in $\R^+$ (or respectively in $\R^-)$ and its length is less than one.
But for all $x,y\ge0$ (respectively, $x,y\le0$) we have $0\le h(x+y)\le h(x)h(y)$, so that for large~$n$,
$\sup_{I_n}h\le \sup_{[-\frac{1}{2},\frac{1}{2}]}{h}$. Thus, $h\in L^1\cap L^\infty$.
This implies $ve^{-|\cdot|}\in L^p(\R)$ for all $1\le p\le \infty$.

\begin{theorem}
 \label{cor1}
Let $2\le p\le \infty$ and $\phi$ be  a $v$-moderate weight function as in Definition~\ref{defaw} satisfying
condition~\eqref{vadmp} instead of~\eqref{vadm}.
Let also 
$u|_{t=0}=u_0$ satisfy
  \[
  \begin{cases}
  u_0\,\phi\in L^p(\R)\\
  u_0\phi^{1/2}\in L^2(\R)
  \end{cases}
  \qquad\hbox{and}\qquad
  \begin{cases}
   (\partial_xu_0)\phi\in L^p(\R),\\
   (\partial_x u_0)\phi^{1/2}\in L^2(\R).
   \end{cases}
 \]
Let also $u\in C([0,T],H^s(\R))$, $s>3/2$ be the strong solution
 of the Cauchy problem for  equation~\eqref{CH}, emanating from~$u_0$
 Then, 
 \[
 \sup_{t\in[0,T]}  \Bigl( \| u(t)\phi\|_p+\|\partial_xu(t)\phi\|_p \Bigr) <\infty.
\]
and
 \[
 \sup_{t\in[0,T]}  \Bigl( \| u(t)\phi^{1/2}\|_2+\|\partial_xu(t)\phi^{1/2}\|_2 \Bigr) <\infty.
\]
   \end{theorem}

\begin{remark}
The two auxiliary conditions  $u_0\phi^{1/2}\in L^2$ and $(\partial_xu_0)\phi^{1/2}\in L^2$
in Theorem~\ref{cor1} are usually very easily checked: they are fulfilled if, for example, 
$\phi^{-1}\in L^{(\frac{p}{2})'}$, where $\frac{1}{q}+\frac{1}{q'}=1$.
Hence, these conditions are automatically satisfied
for example when $\phi(x)\ge (1+|x|)^{\alpha}$  for some $\alpha>1$.
Such two conditions are also fulfilled as soon as $u_0$ and $\partial_x u_0$ belong to $L^{p'}$
\end{remark}

\begin{remark}
\label{rempin}
The main motivation for Theorem~\ref{cor1} is to cover the limit case case 
of the weights $\phi=\phi_{1,1,c,d}$ (see~\eqref{exaw}) that were apparently excluded by Theorem~\ref{theo2}.
Theorem~\ref{cor1} tells us that the statement of Theorem~\ref{theo2} is in fact valid
for such weights provided, for example, $c<0$, $d\in\R$ and $\frac{1}{|c|}<p\le\infty$.

We can also apply this theorem choosing $\phi(x)=\phi_{1,1,0,0}(x)=e^{|x|}$ and $p=\infty$.
it follows that if $|u_0(x)|$ and $|\partial_x u_0(x)|$ are both bounded by $c\,e^{-|x|}$,
then the strong solution satisfies, uniformly in $[0,T]$,
\begin{equation}
\label{expdecay}
|u(x,t)|+ |\partial_x u(x,t)|\le Ce^{-|x|}.
\end{equation}
\end{remark}

The peakon-like decay~\eqref{expdecay}
is the fastest possible decay that is possible to propagate for a nontrivial solution~$u$.
Indeed, arguing as in the proof of Theorem~{1.1} of \cite{HMPZ07} it is not difficult to
see that, for fast enough decaying data (say, $|u_0(x)|+|\partial_xu_0(x)|\le Ce^{-a|x|}$, with $a>1/2$),
the following asymptotic profiles hold:
\begin{equation}
\label{a52}
\begin{split}
&u(x,t)\sim u_0(x)+e^{-x}t\,\Phi(t), \qquad\hbox{as $x\to+\infty$}\\
&u(x,t)\sim u_0(x)-e^{x}t\,\Psi(t), \qquad\hbox{as $x\to-\infty$}
\end{split}
\end{equation}
where $\Phi(t)\not=0$ and $\Psi(t)\not=0$ for all~$t\in[0,T]$ (unless $u\equiv0$). 
Thus, if $u_0(x)=o(e^{-|x|})$, only the zero solution can decay faster than $e^{-|x|}$
at a later time $0<t_1\le T$.
We will prove the validity of the asymptotic profile~\eqref{a52} (under slightly weaker conditions)
 for the reader's convenience at the end of this section using the ideas of~\cite{HMPZ07}.

\begin{proof}[Proof of Theorem~\ref{cor1}]
We start observing that $\phi^{1/2}$ is a $v^{1/2}$-moderate weight such that
$|(\phi^{1/2})'(x)|\le \frac A2\phi^{1/2}(x)$. Moreover, $\inf_\R v^{1/2}>0$.
By condition~\eqref{vadmp}, $v^{1/2}e^{-|x|/2}\in L^{2p}(\R)$, hence H\"older's inequality
implies that $v^{1/2}e^{-|x|}\in L^{1}(\R)$.
Then Theorem~\ref{theo2} applies with $p=2$ to the weight $\phi^{1/2}$ yielding
\begin{equation}
\| u(t)\phi^{1/2}\|_2+\|(\partial_xu)(t)\phi^{1/2}\|_2
  \le \Bigl(\| u_0\phi^{1/2}\|_2+\|(\partial_xu_0)\phi^{1/2}\|_2\Bigr)e^{CMt}.
\end{equation}
Setting as before $F(u)=u^2+\frac 12 (\partial_x u)^2$ this implies
\begin{equation}
\label{estF}
\|F(u)(t)\phi\|_1\le K_0\,e^{2CMt}.
\end{equation}
The constants $K_0$ and $K_1$ below depend only on $\phi$ and on the datum.

Arguing as in the proof of Theorem~\ref{theo2} we obtain, for $p<\infty$, the inequalities
(recall that $f(x)=f_N(x)=\min\{\phi(x),N\}$):
\begin{equation}
\label{ppeb}
\begin{split}
\frac{d}{dt}  \|uf\|_{p} 
 &\le M\|uf\|_{p}  +\|f(\partial_xG*F(u))\|_{p}\\
\end{split}
\end{equation}
and
\begin{equation}
\label{2peb}
\begin{split}
\frac{d}{dt}  \|(\partial_xu)f\|_{p} 
 &\le M(2+A)\|(\partial_x u)f\|_{p}  +\|f(\partial^2_xG*F(u))\|_{p}\\
\end{split}
\end{equation}
Recall that $|\partial_x G|\le \frac{1}{2}e^{-|\cdot|}$. 
Then combining Proposition~\ref{wyoun2} with condition~\eqref{vadmp}
and estimate~\eqref{estF}, we get
\begin{equation*}
\begin{split}
\|f(\partial_xG*F(u))\|_{p}
\le K_1\,e^{2CMt}.
\end{split}
\end{equation*}
The constant in the right-hand side is independent on~$N$.
Similarly, recalling that $\partial^2_xG=G-\delta$, 
\begin{equation*}
\begin{split}
\|f(\partial^2_xG*F(u))\|_{p}
&\le 
\|f(G*F(u))\|_{p}+\|fF(u)\|_p\\
&\le K_1\,e^{2CMt}+M \bigl(\|uf\|_p +\|(\partial_x u)f\|_p \bigr),
\end{split}
\end{equation*}
Plugging the two last estimates in~\eqref{ppeb}-\eqref{2peb}, 
and summing we obtain
\[
 \frac{d}{dt}\Bigl( \|u(t)f\|_p+\|(\partial_x u)(t)f\|_p\Bigr)
 \le  M(3+A)\Bigl( \|u(t)f\|_p+\|(\partial_x u)(t)f\|_p\Bigr) +2K_1\,e^{2CMt}
\]

Integrating 
and finally letting $N\to\infty$
yields the conclusion in the case $2\le p<\infty$. 
The constants throughout the proof are independent on~$p$.
Therefore, for $p=\infty$ one can rely on the result established for
finite exponents~$q$ and then let~$q\to\infty$.
The argument is fully similar to that of Theorem~\ref{theo2}.
\end{proof}

\medskip
We finish this section with the proof of the asymptotic profile~\eqref{a52}.
This can be seen as an extension of Theorem~{1.4} of~\cite{HMPZ07},
where the assumption of compactly supported data of \cite{HMPZ07} 
is relaxed by putting a more general decay condition at infinity.
 
\begin{corollary}
\label{propuc}
Let $s>3/2$ and $u_0\in H^s$, $u_0\not\equiv0$,
such that 
\begin{equation}
\label{ex12}
\sup_{x\in\R} e^{|x|/2} (1+|x|)^{1/2}\log(e+|x|)^{d}\Bigl(|u_0(x)|+|(\partial_xu_0)(x)| \Bigr)<\infty,
\end{equation}
for some $d>1/2$.
Then condition~\eqref{ex12} is conserved uniformly in $[0,T]$ by the strong solution $u\in C([0,T], H^s)$ 
of the Camassa--Holm equation. Moreover,
the following asymptotic profiles (respectively for $x\to+\infty$ and $x\to-\infty)$ hold:
\begin{equation}
\label{prof2}
\begin{cases}
u(x,t)=u_0(x)+e^{-x}t\bigl[\Phi(t)+\epsilon(x,t)\bigr], 
&\quad\hbox{with}\quad
 \lim_{x\to+\infty} \epsilon(x,t)=0,\\
u(x,t)=u_0(x)-e^{x}t\bigl[\Psi(t)+\varepsilon(x,t)\bigr],
&\quad\hbox{with}\quad
 \lim_{x\to-\infty} \varepsilon(x,t)=0,
\end{cases}
\end{equation}
where,  for all $t\in[0,T]$ and some constants $c_1,c_2>0$ independent on~$t$,
\begin{equation}
 \label{c120}
c_1\le \Phi(t)\le c_2,\qquad
c_1\le \Psi(t)\le c_2.
\end{equation}
\end{corollary}

\begin{proof}
The fact that condition~\eqref{ex12} 
is conserved uniformly in $[0,T]$ by the strong solution $u\in C([0,T], H^s)$
is a simple application of Theorem~\ref{cor1},
with $p=\infty$ and the weight
$\phi(x)=e^{|x|/2}(1+|x|)^{1/2}\log(e+|x|)^{d}$ .
Integrating Equation~\eqref{CH2} we get
\begin{equation}
\label{duh}
u(x,t)=u_0(x)-\int_0^t\partial_x G*F(u)(x,s)\,ds-\int_0^t u\partial_xu(x,s)\,ds.
\end{equation}
We denoted, as before, $F(u)=u^2+\frac12(\partial_x u)^2$.

We have, for $t\in[0,T]$,
\[
 \biggl|\int_0^t u(x,s)\partial_xu(x,s)\,ds\biggr|\le Ce^{-|x|}\,t(1+|x|)^{-1}\log(e+|x|)^{-2d}. 
\]
Then the last term in~\eqref{duh} can be included inside the lower order terms
of the asymptotic profiles~\eqref{prof2}.

For $0<t\le T$ let us  set 
\[
 h(x,t)=\frac{1}{t}\int_0^tF(u)(x,s)\,ds
 \]
and
\begin{equation*}
 \begin{split}
  \Phi(t)=\frac{1}{2}\int_{-\infty}^\infty e^{y}h(y,t)\,dy, \qquad
  \Psi(t)=\frac{1}{2}\int_{-\infty}^\infty e^{-y}h(y,t)\,dy.
 \end{split}
\end{equation*}
The function $(1+|\cdot|)^{-1/2}\log(e+|\cdot|)^{-d}$ belongs to $L^2(\R)$. 
Then applying Theorem~\ref{cor1}, now  with $p=2$ and the weight $\phi(x)=e^{|x|/2}$ yields 
$\int e^{|y|}h(y,t)\,dy<\infty$.
We extend by continuity the definition of~$\Phi$ and $\Psi$ at $t=0$,
setting $\Phi(0)=\frac12\int_{-\infty}^\infty e^yF(u_0)(y)\,dy$
and $\Psi(0)=\frac12\int_{-\infty}^\infty e^{-y}F(u_0)(y)\,dy$.
The assumption
$u_0\not\equiv0$, ensures the validity of estimates~\eqref{c120} with $c_1,c_2>0$ for all $t\in[0,T]$.

Now 
using $\partial_xG(x-y)=-\frac{1}{2}{\rm sign}(x-y)e^{-|x-y|}$ we get
\[
\begin{split}
-\int_0^t\partial_x G*F(u)(s)\,ds
&=\frac{e^{-x}t}{2}\int_{-\infty}^x e^yh(y,t)\,dy-\frac{e^x t}{2} \int_x^\infty e^{-y}h(y,t)\,dy\\
&=e^{-x}t\biggl[\Phi(t) - \frac{1}{2}\int_x^{+\infty}\bigl(e^y+e^{2x-y}\bigr)h(y,t)\,dy\biggr]. \\
\end{split}
\]
But
\[
0\le \int_x^{+\infty}\bigl(e^y+e^{2x-y}\bigr) h(y,t)\,dy
 \le 2\int_x^\infty e^yh(y,t)\,dy\,\longrightarrow0, \qquad\hbox{as $x\to+\infty$,}
\]
by the dominated convergence theorem.
This proves the first asymptotic profile of~\eqref{prof2}.

In the same way, 
\[
\begin{split}
-\int_0^t\partial_x G*F(u)(s)\,ds
&=-e^{x}t\biggl[\Psi(t) - \frac{1}{2}\int_{-\infty}^{x}\bigl(e^{-y}+e^{y-2x}\bigr)h(y,t)\,dy\biggr] \\
\end{split}
\]
and
\[
0\le \int_{-\infty}^x\bigl(e^{-y}+e^{y-2x}\bigr) h(y,t)\,dy
 \le 2\int_{-\infty}^x e^{-y}h(y,t)\,dy\,\longrightarrow0, \qquad\hbox{as $x\to-\infty$.}
\]
This establish the second asymptotic profile of~\eqref{prof2}.
\end{proof}

Applying Corollary~\eqref{propuc} we immediately recover 
the results on the unique continuation for the Camassa--Holm equation
in~\cite{ACon05}, \cite{Hen05}, \cite{HMPZ07}:  for example our profiles~\eqref{prof2} tell us
that only the zero solution can be compactly supported (or decay faster than $e^{-|x|}$)
at two different times $t_0$ and $t_1$.


\bigskip
\section{Breakdown of solutions of the Camassa--Holm equation}

Our next goal is to establish the following blowup criterion based on decay properties.
We show that initial data decaying faster (even if only in a weak sense) than the peakons
 $u_c(x,t)=ce^{-|x-ct|}$ lead to a wave breaking effect.
Our conclusion  (Theorem~\ref{theo1} below) 
strengthens that of~\cite{HMPZ07}, where the authors 
proved that compactly supported data
lead to a breakdown of the solution.

\begin{theorem}
 \label{theo1}
 Let $u_0\not\equiv0$ be such that $u_0\in H^s(\R)$, $s>5/2$, and
 \begin{equation}
  \liminf_{|x|\to\infty} \,e^{|x|}\bigl( |u_0(x)| + |\partial_xu_0(x)|\bigr)=0.
 \end{equation}
 Then the unique strong solution~$u$ 
 of the Camassa--Holm equation~\eqref{CH} such that $u|_{t=0}=u_0$
breaks down in finite time. More precisely, if $u\in C([0,T^*),H^s)$, then $T^*<\infty$
and at the maximal time $T^*$ one has 
 $\int_0^{T^*}\|\partial_x u(t)\|_\infty\,dt=+\infty$.

In particular, a breakdown occurs when the initial datum satisfies
\[ \Bigl \|e^{|x|} \bigl(|u_0|+|\partial_x u_0|\bigr)\, \Bigr\|_{L^p}<\infty \]
for some $1\le p<\infty$.
\end{theorem}

\begin{proof}[Proof of Theorem~\ref{theo1}]
We refine the argument used in~\cite{HMPZ07}, relying on the study of the sign of the difference between
the solution and its second derivative, $m(x,t)=u(x,t)-\partial_x^2u(x,t)$. Let $m_0(x)=m|_{t=0}(x)$.
By contradiction, assume that we have~$T^*=\infty$. Applying a theorem of McKean~\cite{McKean04}
(see also \cite{JiaNZxx} for a recent alternative proof of McKean's theorem and \cite{ACon00} 
for a previous partial result) we must have
\begin{enumerate}
 \item[i)] Either that $m_0(x)$ is of constant sign in~$\R$,
 \item[ii)] Or $m_0$ changes sign in~$\R$ and $\exists\, x_0\in\R$ such that
 \begin{equation*}
  \begin{cases}
   m_0(x)\le0 &\hbox{for $x\le x_0$},\\ m_0(x)\ge0 &\hbox{for $x\ge x_0$}.
  \end{cases}
 \end{equation*}
\end{enumerate}
Let us show that both (i) and (ii) do not hold.

\medskip
From our assumption we can find two real sequences $(a_n)$ and $(A_n)$ such 
that $a_n\to-\infty$ and $A_n\to+\infty$ and such that 
\begin{equation*}
\begin{split}
\int_{a_n}^{A_n}e^xm_0(x)\,dx
&= \int_{a_n}^{A_n}e^x(u_0-\partial_x^2u_0)(x)\,dx\\
&=e^{A_n}(u_0-\partial_xu_0)(A_n) - e^{a_n}(u_0-\partial_xu_0)(a_n)\\
&\quad\longrightarrow0, \qquad\hbox{as $n\to\infty$}.
\end{split}
\end{equation*}
As $m_0\not\equiv0$, this implies that $m_0$ must change sign in~$\R$, and that (i) does not hold.

If (ii) were satisfied, then 
we could find $x_1<x_0$ such that $m_0(x_1)<0$. 
Hence, there exists $\delta>0$
(depending only on the values of $m_0(x)$ in the interval $[x_1,x_0]$)
such that, for $n\to\infty$,
\begin{equation*}
 \begin{split}
  -e^{x_0}\int_{a_n}^{x_0} m_0(x)\,dx
  &\ge -\int_{a_n}^{x_0}e^x m_0(x)\,dx + \delta
  =\int_{x_0}^{A_n} e^x m_0(x)\,dx+\delta+o(1)\\
  &\ge e^{x_0}\int_{x_0}^{A_n}m_0(x)\,dx+\delta+o(1).
 \end{split}
\end{equation*}
Then we get
\begin{equation}
 \label{ine1}
 -\int_{a_n}^{x_0} m_0(x)\,dx\ge \int_{x_0}^{A_n}m_0(x)\,dx +\delta e^{-x_0}+ o(1),
 \qquad\hbox{as $n\to+\infty$}.
\end{equation}

In the same way, there exist two sequences $(b_n)$ and $(B_n)$,
such that $b_n<a_n$ and $B_n<A_n$, with  $B_n\to+\infty$ and such that
\begin{equation*}
\begin{split}
\int_{b_n}^{B_n}e^{-x}m_0(x)\,dx
&=-e^{-B_n}(u_0+\partial_xu_0)(B_n) + e^{-b_n}(u_0+\partial_xu_0)(b_n)
\longrightarrow0,
\end{split}
\end{equation*}
as $n\to\infty$.

Moreover, as $n\to+\infty$,
\begin{equation*}
 \begin{split}
  -e^{-x_0}\int_{b_n}^{x_0}m_0(x)\,dx
  &\le -\int_{b_n}^{x_0}e^{-x}m_0(x)\,dx   =\int_{x_0}^{B_n} e^{-x}m_0(x)\,dx+o(1)\\
  &\le e^{-x_0}\int_{x_0}^{B_n}m_0(x)\,dx+o(1).
 \end{split}
\end{equation*}
Thus,
\begin{equation}
 \label{ine2}
 -\int_{b_n}^{x_0}m_0(x)\,dx\le \int_{x_0}^{B_n} m_0(x)\,dx+o(1), \qquad\hbox{as $n\to+\infty$}.
\end{equation}
Combining inequalities~\eqref{ine1}-\eqref{ine2} we get,
as $n\to+\infty$,
\begin{equation}
\begin{split}
 \delta e^{-x_0}+\int_{x_0}^{A_n} m_0(x)\,dx+o(1)
 &\le -\int_{a_n}^{x_0} m_0(x)\,dx \le -\int_{b_n}^{x_0} m_0(x)\,dx\\
 &=\int_{x_0}^{B_n} m_0(x)\,dx+o(1) \le \int_{x_0}^{A_n} m_0(x)\,dx+o(1).
\end{split}
\end{equation}

Letting $n\to+\infty$ we contradict the fact that $\delta>0$. 
\end{proof}

\bigskip

\section{Acknowledgements}
The author would like to thank Yong Zhou for hosting him in Zhejiang Normal University and for the many
interesting discussions on the topic of the present paper.
The author is also grateful to the referee for his observations that helped him to
improve the first version.


\bibliographystyle{alpha}

\begin{bibdiv}
\begin{biblist}

\end{biblist}
\end{bibdiv}


\begin{thebibliography}{11}
\addcontentsline{toc}{chapter}{Bibliography}

\bib{AldG01}{article}{
   author={Aldroubi, A.},
   author={Gr{\"o}chenig, K.},
   title={Nonuniform sampling and reconstruction in shift-invariant spaces},
   journal={SIAM Rev.},
   volume={43},
   date={2001},
   number={4},
   pages={585--620 (electronic)},
  }



		

\bib{BreC07}{article}{
   author={Bressan, A.},
   author={Constantin, A.},
   title={Global conservative solutions of the Camassa-Holm equation},
   journal={Arch. Ration. Mech. Anal.},
   volume={183},
   date={2007},
   number={2},
   pages={215--239},
}


\bib{BreC07bis}{article}{
   author={Bressan, A.},
   author={Constantin, A.},
   title={Global dissipative solutions of the Camassa-Holm equation},
   journal={Anal. Appl. (Singap.)},
   volume={5},
   date={2007},
   number={1},
   pages={1--27},
}


\bib{CamH93}{article}{
  author={Camassa, R.},
  author={Holm, L.},
  title={An integrable shallow--water equation with peaked solitons},
  journal={Phys. Rev. Letters},
  volume={71},
  year={1993},
  pages={1661--1664},
  } 


\bib{ACon00}{article}{
   author={Constantin, A.},
   title={Existence of permanent and breaking waves for a shallow water equation: A geometric approach},
   journal={Ann. Inst. Fourier (Grenoble)},
   volume={50},
   year={2000},
   pages={321--362},
    }


\bib{ACon05}{article}{
   author={Constantin, A.},
   title={Finite propagation speed for the Camassa--Holm equation},
   journal={J. Math. Phys.},
   volume={46},
   year={2005},
    }

\bib{AConE07}{article}{
   author={Constantin, A.},
   author={Escher, J.},
   title={Particle trajectories in solitary water waves},
   journal={Bull. Amer. Math. Soc. (N.S.)},
   volume={44},
   date={2007},
   number={3},
   pages={423--431 (electronic)},
}


\bib{AConL09}{article}{
   author={Constantin, A.},
   author={Lannes, D.},
   title={The hydrodynamical relevance of the Camassa-Holm and
   Degasperis-Procesi equations},
   journal={Arch. Ration. Mech. Anal.},
   volume={192},
   date={2009},
   number={1},
   pages={165--186},
  }

\bib{AConS00}{article}{
   author={Constantin, A.},
   author={Strauss, W. A.},
   title={Stability of peakons},
   journal={Comm. Pure Appl. Math.},
   volume={53},
   date={2000},
   number={5},
   pages={603--610},
}


\bib{Dan01}{article}{
   author={Danchin, R.},
   title={A few remarks on the Camassa-Holm equation},
   journal={Differential Integral Equations},
   volume={14},
   date={2001},
   number={8},
   pages={953--988},
}


\bib{Dan03}{article}{
   author={Danchin, R.},
   title={A note on well-posedness for Camassa-Holm equation},
   journal={J. Differential Equations},
   volume={192},
   date={2003},
   number={2},
   pages={429--444},
}


\bib{DeLKT11}{article}{
   author={de Lellis, C.},
   author={Kappeler, T.},
   author={Topalov, P.},
   title={Low-regularity solutions of the periodic Camassa-Holm equation},
   journal={Comm. Partial Differential Equations},
   volume={32},
   date={2007},
   number={1-3},
   pages={87--126},
}

\bib{EKPV1}{article}{
   author={Escauriaza, L.},
   author={Kenig, C. E.},
   author={Ponce, G.},
   author={Vega, L.},
   title={On uniqueness properties of solutions of the $k$-generalized KdV
   equations},
   journal={J. Funct. Anal.},
   volume={244},
   date={2007},
   number={2},
   pages={504--535},
}


\bib{EKPV2}{article}{
   author={Escauriaza, L.},
   author={Kenig, C. E.},
   author={Ponce, G.},
   author={Vega, L.},
   title={On uniqueness properties of solutions of Schr\"odinger equations},
   journal={Comm. Partial Differential Equations},
   volume={31},
   date={2006},
   number={10-12},
   pages={1811--1823},
}




  
\bib{Fei79}{article}{
   author={Feichtinger, H. G.},
   title={Gewichtsfunktionen auf lokalkompakten Gruppen},
   language={German},
   journal={\"Osterreich. Akad. Wiss. Math.-Natur. Kl. Sitzungsber. II},
   volume={188},
   date={1979},
   number={8-10},
   pages={451--471},
   }


\bib{FokF81}{article}{
    author={Fokas, A.S.},
    author={Fuchssteiner, B.},
    title={Symplectic structures, their B\"acklund transformation and hereditary symmetries},
    journal={Physica D.},
    volume={4},
    date={1981},
    pages={47--66},
    }

\bib{Gro07}{article}{
   author={Gr{\"o}chenig, K.},
   title={Weight functions in time-frequency analysis},
   conference={
      title={Pseudo-differential operators: partial differential equations
      and time-frequency analysis},
   },
   book={
      series={Fields Inst. Commun.},
      volume={52},
      publisher={Amer. Math. Soc.},
      place={Providence, RI},
   },
   date={2007},
   pages={343--366},
}


\bib{Hen05}{article}{
   author={Henry, D.},
   title={Compactly supported solutions of the Camassa-Holm equation},
   journal={J. Nonlinear Math. Phys.},
   volume={12},
   date={2005},
   number={3},
   pages={342--347},
}



\bib{HimM01}{article}{
   author={Himonas, A. A.},
   author={Misio{\l}ek, G.},
   title={The Cauchy problem for an integrable shallow-water equation},
   journal={Differential Integral Equations},
   volume={14},
   date={2001},
   number={7},
   pages={821--831},
}


\bib{HMPZ07}{article}{
   author={Himonas, A. A.},
   author={Misio{\l}ek, G.},
   author={Ponce, G.},
   author={Zhou, Y.},
   title={Persistence properties and unique continuation of solutions of the
   Camassa-Holm equation},
   journal={Comm. Math. Phys.},
   volume={271},
   date={2007},
   number={2},
   pages={511--522},
}


\bib{JiaNZxx}{article}{
  author={Jiang, Z.},
  author={Ni, L.},
  author={Zhou, Y.},
  title={Wave breaking for the Camassa--Holm equation},
  journal={J. Nonlinear Math. Phys.},
  date={to appear},
}

\bib{Kou99}{article}{
   author={Kouranbaeva},
   title={The Camassa--Holm equation as geodesic flow on the diffeomorphism group},
   journal={J. Math. Phys.},
   volume={40},
   year={1999},
   pages={857--868},
    }


\bib{McKean04}{article}{
   author={McKean, H.P.},
   title={Breakdown of the Camassa-Holm equation},
   journal={Comm. Pure Appl. Math.},
   volume={57},
   date={2004},
   number={3},
   pages={416--418},  
}


\bib{Mol04}{article}{
   author={Molinet, Luc},
   title={On well-posedness results for Camassa-Holm equation on the line: a
   survey},
   journal={J. Nonlinear Math. Phys.},
   volume={11},
   date={2004},
   number={4},
   pages={521--533},
}



\bib{NiZ11}{article}{
   author={Ni, L.},
   author={Zhou, Y.},
   title={A New Asymptotic Behavior of Solutions to the Camassa--Holm equation},
   journal={Proc. Amer. Math. Soc.},
   volume={140},
   date={2012},
   pages={607--614},
}


\bib{Tol96}{article}{
   author={Toland, J. F.},
   title={Stokes waves},
   journal={Topol. Methods Nonlinear Anal.},
   volume={7},
   date={1996},
   number={1},
   pages={1--48},
}


\bib{XinZ00}{article}{
   author={Xin, Z.},
   author={Zhang, P.},
   title={On the weak solutions to a shallow water equation},
   journal={Comm. Pure Appl. Math.},
   volume={53},
   date={2000},
   number={11},
   pages={1411--1433},
 }

\bib{XinZ02}{article}{
   author={Xin, Z.},
   author={Zhang, P.},
   title={On the uniqueness and large time behavior of the weak solutions to
   a shallow water equation},
   journal={Comm. Partial Differential Equations},
   volume={27},
   date={2002},
   number={9-10},
   pages={1815--1844},
}




\bib{Yin04}{article}{
   author={Yin, Z.},
   title={Well-posedness, global solutions and blowup phenomena for 
   	a nonlinearly dispersive wave equation},
   journal={J. Evol. Equ.},
   volume={4},
   date={2004},
   pages={391--419},
}


\end{thebibliography}

\end{document}